\documentclass[12pt,oneside]{amsart}

\usepackage{fullpage}
\usepackage{amsfonts,amsmath,amscd,amssymb}
\usepackage{dsfont,mathtools}
\usepackage{xspace}
\usepackage{graphicx}
\usepackage{tikz-cd}
\usepackage{color}
\usepackage{enumitem}
\input{xy}
\xyoption{all}
\xyoption{curve}

\newcommand{\bC}{{\mathbb C}}

\newcommand{\bR}{{\mathbb R}}
\newcommand{\bS}{{\mathbb S}}

\newcommand{\bZ}{{\mathbb Z}}

\newcommand{\fg}{{\mathfrak g}}
\newcommand{\wfg}{{\widetilde{\mathfrak g}}}
\newcommand{\fh}{{\mathfrak h}}
\newcommand{\wfh}{{\widetilde{\mathfrak h}}}
\newcommand{\fk}{{\mathfrak k}}

\newcommand{\Lie}{{{\mathcal Lie}}}
\newcommand{\Red}{{{\mathcal Red}}}

\newcommand{\sC}{{\mathcal C}} 
 
\newcommand{\sG}{{\mathcal G}} 
\newcommand{\sH}{{\mathcal H}}
\newcommand{\sI}{{\mathcal I}}
\newcommand{\sK}{{\mathcal K}}

\newcommand{\sT}{{\mathcal T}}

\newcommand{\sO}{{\mathcal O}} 
 
\newcommand{\sR}{{\mathcal R}} 
\newcommand{\sV}{{\mathcal V}}
\newcommand{\sU}{{\mathcal U}}

\newcommand{\Ind}{{\mathrm{Ind}}}

\newcommand{\aut}{{\mathrm{Aut}}_{\otimes}}

\newcommand{\SU}{{\mathrm{SU}}} 
\newcommand{\SL}{{\mathrm{SL}}} 
 
\newcommand{\GL}{{\mathrm{GL}}} 
\renewcommand{\mod}{{\mathrm{Mod}}} 

\newcommand{\im}{{\mathrm{im}}}

\newtheorem{theorem}{Theorem}[section]
\newtheorem{prop}[theorem]{Proposition}
\newtheorem{lemma}[theorem]{Lemma}
\newtheorem{cor}[theorem]{Corollary}

\begin{document}

\title{Compact Lie Groups and Complex Reductive Groups}
\author{John Jones}
\email{jdsjones200@gmail.com}
\author{Dmitriy Rumynin}
\email{D.Rumynin@warwick.ac.uk}
\author{Adam Thomas}
\email{Adam.R.Thomas@warwick.ac.uk}
\address{Department of Mathematics, University of Warwick, Coventry, CV4 7AL, UK}
\date{August 19, 2022}

\begin{abstract}
  We show that the categories of compact Lie groups and complex reductive groups (not necessarily connected)
  are homotopy equivalent topological categories.
In other words, the corresponding categories enriched in the homotopy category of topological spaces are equivalent.
This can also be interpreted as an equivalence of infinity categories.
\end{abstract}

\maketitle

In this note we examine the categories
\[
\Lie = \{ 
\mbox{\textnormal{Compact Lie Groups}}  \}
\ \mbox{ and } \ 
\Red = \{
{\textnormal{Complex Reductive Groups}} \}.
\]
The morphisms are group homomorphisms, preserving the additional structures.
The categories $\Lie$ and $\Red$ 
are {\em not equivalent}, although the isomorphism classes are in one-to-one correspondence.
For instance, the groups $\SU_2$ and  $\SL_2 (\bC)$ correspond to each other but have different automorphisms.

The aim of this note is to use Homotopy Theory to fix the aforementioned lack of equivalence.
The following result is the main theorem.
\begin{theorem}\label{thm1}
There exists a complexification functor
$\sT : \Lie 
\longrightarrow
\Red$
with the following four properties.
\begin{enumerate}[label=\normalfont(\roman*)]
\item The group $G\in\Lie$ is a maximal compact subgroup of $\sT (G)$.
\item The functor $\sT$ is injective on morphisms. 
\item The functor $\sT$ is essentially surjective on objects.
\item For all $H,G\in\Lie$ the embedding $\sT_{H,G} : \Lie (H,G) \rightarrow \Red (\sT(H),\sT(G))$ is a
homotopy equivalence.  Here the hom-sets are equipped with the compact-open topology.
\end{enumerate}
\end{theorem}
Both $\Lie$ and $\Red$ are topological categories, as in \cite[1.1.1.6]{Lur}.
Theorem~\ref{thm1} shows $\sT$ is a (weak) equivalence of topological categories in the sense of \cite[1.1.3.6]{Lur}.
Translated into the language of $\infty$-categories, cf. \cite[{\S}1.1.4, {\S}1.1.5]{Lur}, \cite[1.1]{RuTa},
this gives an equivalence between the $\infty$-categories determined by $\Lie$ and $\Red$.

Our secondary aim is to investigate how the complexification functor behaves on subgroups.
Suppose $\sG=\sT(G)$. Consider the sets
\[ \bS(G) \coloneqq \{ \textnormal{compact subgroups of } G \}, \ \ 
  \bS(\sG) \coloneqq
  \{ \textnormal{reductive subgroups of } \sG \}.\]


We prove the following result, which refines and extends \cite[Lemma 2.1]{BLS} to not necessarily connected subgroups.
\begin{theorem}\label{thm2}
  Suppose $G\in\Lie$ and $\sG=\sT(G)\in\Red$ is its complexification.
  The complexification functor  gives a function
$\tau: \bS(G) 
\rightarrow
\bS(\sG)$
satisfying the following properties, where $H\in\bS(G)$. 
\begin{enumerate}[label=\normalfont(\roman*)]
\item The group $H$ is a maximal compact subgroup of $\tau (H)$.
\item The function $\tau$ yields a bijection of orbit sets  
  $\overline{\tau}: \bS(G)/G 
  \rightarrow
  \bS( \sG) / \sG$ where $G$, $\sG$ act on $\bS(G)$, $\bS(\sG)$ (respectively) by conjugation.
\item The manifold $G/H$ is homotopy equivalent to the algebraic variety $\sG / \tau (H)$.
\item $H$ is the centraliser of an involution if and only if $\tau (H)$ is the centraliser of an involution.
\end{enumerate}
\end{theorem}

An example, easy to understand, is the algebraic torus $\sT^n = (\bC^{\times})^n$. A maximal compact subgroup is the topological torus
$T^n = \bR^n / \bZ^n$. The hom-spaces  are discrete spaces
\[ \Lie (T^m, T^n ) = \Red (\sT^m, \sT^n ) = \bZ^{n\times m}\]
that agree on the nose. 
For a slightly deeper example, let $\sH=\SL_2 (\bC)$ and let $\sG$ be a simply-connected simple group.
There is a natural $\sG$-equivariant map to the unipotent cone of $\sG$:
\[ \beta: \Red (\sH, \sG) \longrightarrow \sU \subseteq \sG, \ \ f \mapsto  f \left( \begin{pmatrix} 1 & 1 \\ 0 & 1 \end{pmatrix}\right) \, .\]
By the Jacobson-Morozov theorem, $\beta$ yields a bijection between the $\sG$-orbits on $\Red (\sH, \sG)$ and $\sU$. But the topology
of the two spaces is different. The cone $\sU$ is connected and the orbits are locally closed in it.
The orbits in $\Red (\sH, \sG)$ are open: they are connected components.
For $x\in \sU$, let $\sG(x)$ and $\sG^{red}(x)$ be its centraliser and the reductive part of the centraliser, respectively.
The orbit of $x$ is a quasiaffine variety $\sG/\sG(x)$.
Its inverse image under $\beta$ is an affine variety $\sG/\sG^{red}(x)$.
The corresponding component of $\Lie (H,G)$ is 
a compact $G$-orbit $G/K$ inside $\sG/\sG^{red}(x)$, where $K$ is a maximal compact subgroup of $\sG^{red}(x)$.


 In Section~1 we prove Theorem~\ref{thm1}
and in Section~2 we prove Theorem~\ref{thm2}. 
The final section is reserved for comments.

\section{Properties of Complexification}
For $G \in \Lie$ consider the tensor category $\sC_G = G\mbox{-}\mod$ 
of finite-dimensional
continuous 
complex representations of $G$.
Note that any continuous homomorphism of Lie groups is differentiable, so a homomorphism of Lie groups. 
Thus, $\sC_G$ is also the category of differentiable representations.
Consider the forgetful functor
$F_G:\sC_G\rightarrow \sV_\bC$, where $\sV_\bC$ is the category of finite-dimensional complex vector spaces. 
We define {\em the complexification functor} $\sT: \Lie \rightarrow \Red$ via the Tannaka reconstruction, that is $\sT (G) \coloneqq \aut (F_G)$.
In other words, an element of $\sT (G)$ is a family of linear maps
\[ (x_V)_{V\in \sC_G}, \ x_V \in \hom_{\bC} (V,V)\]
defined for each representation $V$ of $G$ satisfying the following three properties
\begin{equation} \label{cond_1}
fx_V = x_Wf , \ \   x_{V\otimes W} = x_{V}\otimes x_{W} \ \mbox{ and } \ x_{I} = \text{Id}_I.
\end{equation}
Here  $f\in\sC_G(V,W)$ and $I$ is the trivial representation.
A continuous homomorphism $f:H\rightarrow G$ defines a pull-back functor $f^\ast: \sC_G \rightarrow \sC_H$,
which allows us to define $\sT$ on morphisms
\begin{equation}\label{cond_mor}
  \aut (F_H) \rightarrow \aut (F_G), \ \ (x_U)_{U\in \sC_H} \mapsto (x_{f^{\ast}V})_{V\in \sC_G} \, .
\end{equation}

Let $A(G)$ be the Hopf algebra of matrix elements on $G$. It consists of continuous functions
\[ G\rightarrow \bC, \ \ g \mapsto \alpha ( \rho_V (g)) \]
where $(V,\rho_V)$ runs over the objects of $\sC_G$ and $\alpha\in \hom_{\bC}(V,V)^\ast$. 
It is a commutative subalgebra of $C(G,\bC)$, the Banach algebra of continuous $\bC$-valued functions.
Indeed, if $\Theta, \Phi \in A(G)$ are realised on $V,W\in \sC_G$, then
all linear combinations are realised on $V\oplus W$ and the product $\Theta\cdot\Phi$ is realised on $V\otimes W$.
It is a Hopf algebra under 
\[ A(G)\rightarrow  A(G\times G)\xrightarrow{\phi^{-1}}  A(G)\otimes A(G), \ \ \Theta \mapsto \Theta \circ \mu_G \]
where $\mu_G : G \times G \rightarrow G$ is the group multiplication and
\[ \phi\, : \, A(G)\otimes A(G)\rightarrow  A(G\times G), \ \ \phi (\sum_i \Theta_i \otimes \Phi_i) (x,y) =  \sum_i \Theta_i (x) \Phi_i (y)\]
is the algebra homomorphism. A standard proof, which includes decomposing a representation of $G\times G$ into a direct sum 
$\oplus_i V_i \otimes W_i$ where all $V_i$ and $W_i$ are representations of $G$, shows that $\phi$ is an isomorphism.

The relation between $A(G)$ and $\sC_G$ is at  the heart of Tannaka reconstruction.
\begin{prop} \label{TG_algebraic_0}
The category of finite-dimensional $A(G)$-comodules  is equivalent to $\sC_G$ as symmetric monoidal categories.
\end{prop}  
\begin{proof}
  Given a finite dimensional  $A(G)$-comodule $(V,\varpi : V \rightarrow A(G) \otimes V)$, the corresponding representation is
  \[ (V, \rho), \ \ \rho (g) v  = \sum_{(v)} v_{(-1)}(g) v_{(0)} \mbox{ where } \ \varpi (v) = \sum_{(v)} v_{(-1)} \otimes v_{(0)} \, .\]
  In the opposite direction, given $(V, \rho)\in \sC_G$ we choose a basis $e_1,\ldots , e_n$ of $V$ and write the comodule structure explicitly as
  \[ (V, \varpi), \ \ \varpi (v) = \sum_{i} e^i ( \rho (g) v) \otimes e_i  \, .\]
 It is routine to check that these inverse equivalences of symmetric monoidal categories. 
\end{proof}

The following fact, which easily follows from Proposition~\ref{TG_algebraic_0}, is a version of Peter-Weyl Theorem.
\begin{cor} \label{TG_algebraic_1}
Let $V_i$, $i\in \sI$ be a complete set of representatives of irreducible representations in $\sC_G$.
Then $A(G)$ is isomorphic as a coalgebra to $\oplus_{i\in\sI} V_i \otimes V_i^{\ast}$.
\end{cor}  
\begin{proof}
  Since the category $\sC_G$ is semisimple,  the coalgebra $A(G)$ is cosemisimple.
Thus, it must be equal to its coradical, which is isomorphic to $\oplus_{i\in\sI} V_i \otimes V_i^{\ast}$.
\end{proof}

Let us tie up all the loose ends to get a full picture of Tannaka reconstruction.
\begin{prop} \label{TG_algebraic}
The  following statements hold for a compact Lie group $G$.
\begin{enumerate}[label=\normalfont(\roman*)]
\item The group $\sT(G)$ is isomorphic to the group of $\bC$-points of the pro-algebraic group, represented by $A(G)$.
\item The category of rational representations $\sR_{\sT(G)}$ is equivalent to $\sC_G$ as symmetric monoidal categories.
\item $\sT(G)$ is a reductive algebraic group. 
\item $\sT$ is a functor from $\Lie$ to $\Red$. 
\end{enumerate}
\end{prop}  
\begin{proof}
%
%
  Choose $V_i$, $i\in \sI$ as in Corollary~\ref{TG_algebraic_1}. 
  Consider the relation between  $x=(x_V) \in \sT (G)$
  and an algebra homomorphism $\phi: A(G) \rightarrow \bC$ expressed by the formula
  \begin{equation} \label{isom_TG_Gr}
    w^{\ast} (x_{V_i} (v)) = \phi (v \otimes w^{\ast}) \ \ \mbox{ for all } \ \  i\in\sI \ \mbox{ and } \ v \otimes w^{\ast} \in V_i \otimes V_i^{\ast} \, .
    \end{equation}
  Every element $x=(x_V) \in \sT (G)$ determines an algebra homomorphism $\phi$ by this formula.
  This is the homomorphism $\sT (G) \rightarrow \hom (A(G), \bC)$.

  In the opposite direction, every $\phi\in \hom (A(G), \bC)$ yields a collection $(x_i)$ of linear operators $V_i\rightarrow V_i$
  by \eqref{isom_TG_Gr}. This extends to an element $x=(x_V) \in \sT (G)$ by using the canonical decompositions using the coefficient vector
  spaces $\hom (V_i,V)$
\[ V \xrightarrow{\cong} \bigoplus_{i\in\sI} V_i \otimes \hom (V_i,V) \xrightarrow{\oplus_i (x_i \otimes I)} \bigoplus_{i\in\sI} V_i \otimes \hom (V_i,V)\xrightarrow{\cong} V \, .\]
This is the inverse homomorphism $\hom (A, \bC)\rightarrow \sT (G)$.
  Thus,  \eqref{isom_TG_Gr} gives a group isomorphism $\sT (G) \cong\hom (A,\bC)$. This proves Statement (i).

  Statement (ii) follows immediately from Statement (i) and Proposition~\ref{TG_algebraic_0}.

To prove Statement (iii) it is sufficient to show that $A(G)$ is finitely generated. Indeed, then $\sT(G)$ is an affine group scheme and any affine group scheme in characteristic zero is reduced, thus $\sT(G)$ is an algebraic group. It is reductive since the category $\sR_{\sT(G)}$ is semisimple.

Let $G_1$ be the identity component of $G$ and let $U_1, \ldots, U_k$ be irreducible $G_1$-modules whose highest weights generate the cone of dominant weights of $G_1$. Let $W$ be the direct sum of all distinct irreducible direct summands of the induced modules $\Ind_{G_1}^{G}U_i$.
The representation $W$ {\em generates} $\sC_G$,
in the sense that $\sC_G$ is the only full subcategory closed under isomorphisms, tensor products, direct sums and summands that contains $W$.
Hence, a collection $(x_U)_{U\in \sC_H}$ is uniquely determined by $x_W\in \GL (W)$
and $W\otimes W^\ast$ generates the algebra $A(G)$.

Finally, considera a Lie group homomorphism $f: H \rightarrow G$.
It yields the restriction functor $f^\ast: \sC_G \rightarrow \sC_H$
and the restriction of functions $f^{\sharp}: A(G)\rightarrow C(H,\bC)$, which is an algebra homomorphism.
If $\Phi \in A(G)$ is realised on $V\in \sC_G$,
then $f^{\sharp}(\Phi) \in C(H,\bC)$ is realised on $f^\ast(V)\in \sC_H$,
so that $f^{\sharp}$ is an algebra homomorphism $A(G)\rightarrow A(H)$.

Since $f^\ast$ is a monoidal functor, $f^{\sharp}$
is a homomorphism of Hopf algebras, representing the group homomorphism $\sT(f)$. This proves Statement (iv).
\end{proof}

Proposition~\ref{TG_algebraic} links our definition of $\sT$ with classical definitions.
The representation $W$, constructed during the proof of Statement~(ii) is faithful:
the map
\[ \varphi :\sT(G)\longrightarrow \GL(W)\, , \ \ \ (x_V)_{V\in \sC_G}\mapsto x_W \]
identifies $\sT(G)$ with a subgroup of $\GL (W)$.
The image $\im (\varphi)$ is Zariski closed, thus, adorned with an algebraic group structure,
cf. \cite[Theorem 2.11 and Proposition 2.20]{DM}. 
Historically first, Chevalley essentially constructs the Hopf algebra $A$ as the quotient of $\bC [\GL (W)]$  \cite[Ch. VI]{Che}. 
In his turn, Hochschild defines $\sT(G)$ by its universal property and then proves that it is the same subgroup of $\GL (W)$ as above
\cite[Ch. XVII]{Ho0}.

Each consecutive subsection is a proof of the corresponding part of Theorem~\ref{thm1}.
Parts (i) and (iii) are known. They can be deduced from the results in
\cite{Che}
or
\cite{Ho0}
but it is not shorter than the proofs we provide.

\subsection{$G$ is a Maximal Compact Subgroup of $\sT(G)$}
We have a natural injective group homomorphism
\begin{equation} \label{group_hom}
\iota_G :G \rightarrow \sT (G), \ \ g \mapsto (\rho_V(g))_{V}, \ \ V=(V,\rho_V)\in \sC_G. 
\end{equation}  
The topology on $\sT(G)$ is the subspace topology, induced from $\GL (W)$.
Since $\rho_W: G \rightarrow \GL (W)$ is continuous,
the homomorphism $\iota_G$ is continuous too. Thus, $G$ is naturally a compact subgroup of $\sT (G)$.

Suppose that $G$ is not maximal in $\sT(G)$. Then it is properly contained in another compact subgroup $H$. Applying $\sT$ to the embedding
$f: G\hookrightarrow H$, we get the homomorphisms of groups
\[ G \xrightarrow{f} H \xrightarrow{d} \sT(G) \xrightarrow{\sT(f)} \sT(H) \, . \]
In the opposite direction we have the pull-back functors
\[ \sC_G \xleftarrow{f^\ast}\sC_H \xleftarrow{d^\ast} \sR_{\sT(G)} \xleftarrow{\sT(f)^\ast} \sR_{\sT(H)} \, . \]
The Tannakian formalism (Proposition~\ref{TG_algebraic}) tells us that the pull-back functors $(df)^\ast$ and $(\sT(f)d)^\ast$ are equivalences of categories.
Hence, the pull-back functor $f^{\ast}$ is an equivalence of categories $\sC_G \xrightarrow{\simeq} \sC_H$.  

It remains to conclude that $G=H$. One way to do this is to observe that $A(G)=A(H)$.
Then by the Peter-Weyl theorem, $C(G,\bC)= \overline{A(G)}=\overline{A(H)}=C(H,\bC)$.
Then $G=H$ by the Banach-Stone theorem. Note that this argument is essentially 
the Tannaka-Krein Reconstruction \cite{DoRo}.


A purely representation-theoretic argument is possible too.
Consider $L\coloneqq L^2(H/G,\mu)$, where $\mu$ is a non-zero $H$-invariant measure on the quotient space $H/G$. 
Suppose $G\neq H$. Then $L$ contains non-constant functions.
Since every unitary $H$-module is a direct sum (as Hilbert spaces) of finite dimensional $H$-modules
\cite[Th. 7.1.4]{BaRa}, 
$L$  contains a non-trivial finite dimensional simple $H$-submodule $V$. Evaluation at the coset of the identity
\[ t : V \longrightarrow \bC, \ \ \ t ( \Theta (x)) = \Theta ( G) = \Theta (1_H G) \] 
is a non-zero homomorphism of $G$-modules. This yields a contradiction since $V$ and $\bC$ are distinct irreducible $H$-modules,
so $t$ cannot be a homomorphism of $H$-modules.


\subsection{Injective on Morphisms}
The map $\sT_{H,G} : \Lie (H,G) \rightarrow \Red (\sT(H),\sT (G))$ is injective
because, as is clear from Equations \eqref{cond_mor} and \eqref{group_hom}, 
the restriction of $\sT(f):\sT(H)\rightarrow\sT(G)$ to the compact group $H$
is $f:H\rightarrow G \subset \sT(G)$.

\subsection{Homotopy Equivalence on Morphisms}
Let $\sG=\sT(G)$ and $\sH=\sT(H)$.
We need to prove that the embedding $\sT = \sT_{H,G} : \Lie (H,G) \rightarrow \Red (\sH,\sG)$ is a
homotopy equivalence for all $H,G\in\Lie$.
We explain quickly why it holds and check six technical steps in subsections below.

The identity components $G_1$ and $\sG_1$ act on the topological spaces $\Lie (H,G)$ and $\Red (\sH,\sG)$ by conjugation
\[
  g\cdot \phi := \gamma (g)\phi: h \mapsto g\phi (h) g^{-1}, \quad \text{for } \, g \in G_1 \ (\mbox{or }\sG_1), \
  \phi \in \Lie (H,G) \ (\mbox{or } \Red(\sH,\sG)).
\]
The orbits of these two actions are connected components and they are in one-to-one correspondence. Indeed, 
choosing representatives $\phi_i$ for the $G_1$-orbits on $\Lie (H,G)$ we can match-up the components:
\[
\Lie (H,G)=\coprod_{i\in I} G_1 \cdot \phi_i
\ \ \mbox{ and } \ \ 
\Red (\sH,\sG)=\coprod_{i\in I} \sG_1 \cdot \sT(\phi_i) \, .
\]
On each component $G_1\cdot \phi_i$ the map $\sT$ looks like the map of the quotient spaces
\[
\sT_{i} \, : \, G_1/K = G_1 \cdot \phi_i \rightarrow \sG_1 / \sK = \sG_1 \cdot \sT(\phi_i), \ \ g\cdot \phi_1 \mapsto g\cdot \sT (\phi_1)
\]
where $\sK$ is the stabiliser of the image of $\sT(\phi_i)$ and $K$ is the stabiliser of $\phi_i$.
We claim that $K$ is a maximal compact subgroup of $\sK$.
It follows that each $\sT_{i}$ is a homotopy equivalence
and so is $\sT_{H,G} = \coprod_i \sT_{i}$.

\subsubsection{$\sG_1$-orbits are connected components.}
It is an old result of Demazure that the topological space $\Red (\sH,\sG)$
carries a structure of an algebraic variety such that the 
orbit maps
\[ \varpi_{\phi}: \sG \rightarrow \Red (\sH,\sG), \ \
\varphi_{\phi} (g) = \gamma(g)\phi \]
are smooth for all $\phi\in \Red (\sH,\sG)$, cf. a modern take on the topic by Brion \cite[Lemma 3.5]{Bri}. 
It follows that restrictions of the orbit maps to $\sG_1$ are smooth as well.
Hence, the $\sG_1$-orbits are open and connected, so they are connected components of
$\Red (\sH,\sG)$.

\subsubsection{$\sT_{H,G}$ is continuous}
Let us recall the decomposition in 
the Malcev-Iwasawa theorem \cite[Thm 32.5]{Str}.
Let $\wfg$ be the Lie algebra of $\sG$, $\fg$ its Lie subalgebra corresponding to the Lie algebra of $G$. Then
  $\wfg = \fg \oplus i \fg$ as adjoint $G$-modules where multiplication by $i$ is an isomorphism of $G$-modules $\fg\rightarrow i \fg$.
  The decomposition is a diffeomorphism
  \begin{equation} \label{MI_decomp}
    \mu = \mu_G: G \times i\fg \xrightarrow{\cong\ \mbox{\tiny{as spaces}}} \sG, \ \ \ (g,\alpha) \mapsto \exp(\alpha) g \, . 
  \end{equation}
  The multiplication in $\sG$ can be recovered from the data on the left: the multiplication and the adjoint action of $G$,
  the Lie bracket on $\fg$ and the Baker-Campbell-Hausdorff formula.
  
  The similar ingredients for $H$ are $\wfh = \fh \oplus i \fh$ and
  $\mu = \mu_H: H \times i\fh \rightarrow \sH$. 
Let $\phi\in \Lie (H,G)$, $d\phi: \fh \rightarrow \fg$ its differential. The homomorphism  $\sT(\phi)$ can be written explicitly as 
\begin{equation} \label{MI_decomp_phi}
  \sT (\phi) :
  \sH \xrightarrow{\mu^{-1}}
  H \times i \fh \xrightarrow{ \phi \times i d \phi}
  G \times i\fg \xrightarrow{\mu} \sG \, .
  \end{equation}
Consider compact subsets $C_1 \subseteq H$,  $C_2 \subseteq i\fh$ and open subsets  
$U_1 \subseteq G$,  $U_2 \subseteq i\fg$. The sets
\[
\sO (C_1,U_1) = \{ \phi\in \Lie (H,G) \,\mid\, \phi (C_1) \subseteq U_1 \}, \]
\[
\sO (C_1\times C_2 ,U_1\times U_2) = \{ \psi\in \Red (\sH,\sG) \,\mid\, \psi \mu(C_1\times C_2) \subseteq \mu (U_1\times U_2) \}
\]
form a subbase of the compact-open topologies on  $\Lie (H,G)$ and $\Red (\sH,\sG)$.
It remains to observe that, as follows from the formula~\eqref{MI_decomp_phi} and the naturality of $\exp$,
\[
\sT^{-1} (\sO (C_1\times C_2 ,U_1\times U_2)) \supseteq 
\sO (C_1,U_1)  \cap \sO ( \exp (C_2) , \exp^{-1} (U_2)) \, . 
\]

\subsubsection{One-to-one correspondence between $G_1$-orbits and $\sG_1$-orbits}
Since the map $\sT$ is $G$-equivariant, it gives the map of orbit spaces 
\[
\overline{\sT} \, {:} \,  \Lie (H,G)/G_1 \rightarrow \Red (\sH,\sG)/G_1\rightarrow \Red (\sH,\sG)/\sG_1\, , \  \ G_1 \cdot \phi \mapsto \sG_1 \cdot \sT (\phi)
\]
whose bijectivity we need to establish.

To prove surjectivity, notice 
that the image of $\sT$ consists of all such $\psi$
that $\psi (H) \subseteq G$. 
Indeed, $\psi = \sT (\psi_H : H\rightarrow G)$ by the formula~\eqref{MI_decomp_phi}.

Now pick $\psi \in \Red (\sH,\sG)$. The subgroup $\psi (H)$ is compact, so is contained in a maximal compact subgroup $G'\leq \sG$.
Since all maximal compact subgroups are conjugate, there exists $x\in\sG$ such that $xG'x^{-1}= G$.
The Malcev-Iwasawa decomposition implies that $G$ and $\sG$ have the same component group
$\Gamma \coloneqq \pi_0(G) \cong \pi_0 (\sG)$.
It follows that we can write $x=yz$ with $y\in G$ and $z\in \sG_1$.
Then $zG'z^{-1}=G$ and $[\gamma(z)\psi](H)\subseteq G$.
Thus, $\gamma(z)\psi = \sT (\phi)$ for $\phi \coloneqq \gamma(z)\psi|_H \in \Lie (H,G)$
and $\psi \in \sG_1 \cdot \sT (\phi)$.

To prove injectivity, 
consider a $\sG_1$-orbit $\sG_1 \cdot \sT(\phi)$ in the image of $\Lie (H,G)$ and the set
\[X=\{G_1 \cdot \sT(\theta) \,\mid\, G_1 \cdot \sT(\theta)  \subseteq  \sG_1 \cdot \sT(\phi)\} \]
of all $G_1$-orbits in it.
Let  $K=\phi(H_1)$. 
By Lemma~\ref{conjugate_subs}, inside each orbit $x\in X$ we can choose $\psi_x\in x$ such that
$K=\psi_x(H_1)$.
It is known that the group $K Z(\sG_1)$ has finite index in the normaliser $N_{\sG_1} (K)$ \cite[Th. A]{HH}.
Since $g\cdot \psi_x \in x$ for any $g\in K Z(\sG)$, the set $X$ is finite. It follows
that each orbit $x$ is open in $\sG_1 \cdot \sT(\phi)$.
Hence, the set stabiliser
\[ (\sG_1)_x \coloneqq \{ g\in \sG_1 \,\mid \, g\cdot x = x \}\]
of the orbit $x$ is  an open subgroup of finite index. Thus, $(\sG_1)_x = \sG_1$ 
and the set $X$ consists of a single $G_1$-orbit.
This proves injectivity of $\overline{\sT}$. 

\begin{lemma}  
  \label{conjugate_subs}
Suppose $\sG$ is a reductive group, $G$ is a  maximal compact subgroup of $\sG$, $H$ and $K$ are connected closed subgroups of $G$.
If $H$ and $K$ are conjugate in $\sG$, then they are conjugate in $G$.  
\end{lemma}
\begin{proof}
  We denote the Lie algebras of the compact groups by $\fg$, $\fk$ and $\fh$. Thanks to the Malcev-Iwasava decomposition~\eqref{MI_decomp},
  we can find $g\in G$ and $\alpha\in \fg$ such that
  \[K = \exp (i \alpha) g H g^{-1} \exp (-i \alpha)\, . \]

  Replace $H$ with $gHg^{-1}$. Without loss of generality, $K = \exp (i \alpha) H \exp (-i \alpha)$.
  Differentiating the last equality, we get $[ i \alpha, \fh ] \subseteq \fk$. We complete the proof by considering the following
  two mutually exclusive cases.
  \begin{enumerate}
  \item There exists $\beta \in \fh$ such that $[\alpha,\beta]\neq 0$. Then $[i \alpha,\beta]\not\in \fg$, a contradiction. 
  \item For all  $\beta \in \fh$ we have $[\alpha,\beta] = 0$. Since $H$ is connected, it follows that $\exp (i \alpha)$
    acts trivially on $H$ and $H=K$. The lemma is proved.
  \end{enumerate}
\end{proof}

\subsubsection{$G_1$-orbits are connected components.}
We have just proved that $G_1 \cdot \phi = \sT^{-1} (\sG_1 \cdot \sT(\phi))$.
Since $\sG_1$-orbits are open and closed, so are $G_1$-orbits.
Finally, the $G_1$-orbits are connected because $G_1$ is connected.

\subsubsection{$K$ is a maximal compact subgroup of $\sK$}
The stabiliser of a morphism $\phi$ is the centraliser of its image.
Thus, the statement follows from the next lemma.
\begin{lemma} 
  \label{centraliser}
Suppose $\sG$ is a reductive group, $G$ is a maximal compact subgroup of $\sG$, $X\subseteq G$ is a subset.
The centraliser $C_G(X)$ is a maximal compact subgroup of the centraliser $C_{\sG}(X)$.
\end{lemma}
\begin{proof}
  Obviously, both are closed subgroups and $C_G(X)$ is a compact subgroup of $C_{\sG}(X)$.

  Consider $g\in C_{\sG}(X)$. Write it in the Malcev-Iwasawa decomposition $g= \exp(\alpha) h$ with $h\in G$ and $\alpha\in i \fg$.
  For all $x\in X$ we have
  \[\exp(\alpha) h x = gx=xg = x\exp(\alpha) h = \exp(\mbox{ad}(x)(\alpha)) xh \, .\]
  It follows that $h\in  C_{\sG}(X)$ and $\alpha \in \bigcap_{x\in X} \ker (\mbox{ad}(x))$.
  Moreover, this yields a Malcev-Iwasawa-type  decomposition $C_{\sG}(X) \cong C_{G}(X) \times \bR^m$, where
  $\bR^m = \bigcap_{x\in X} \ker (\mbox{ad}(x))$. 

  We can now conclude that $C_{G}(X)$ is maximal compact in $C_{\sG}(X)$.
  Indeed, $C_{G}(X)$ is contained in a maximal compact subgroup $H$ of  $C_{\sG}(X)$.
  Since $C_{G}(X)$ and $H$ are homotopy equivalent, they are equal.
\end{proof}

\subsubsection{Homotopy equivalence}
\label{homotopy_equiv}
Consider a component $G_1\cdot \phi_i$.
The map of quotient spaces
\[
\widetilde{\sT_{i}} \, : \, G_1/K_1  \longrightarrow \sG_1 / \sK_1, \ \ g\cdot \phi_1 \mapsto g\cdot \sT (\phi_1)
\]
is a homotopy equivalence \cite[Lemma 3.5]{RuTa} 
(note that for manifolds, a weak homotopy equivalence is  a homotopy equivalence). 
The finite group $\Gamma_i  \coloneqq \pi_0(K) \cong \pi_0 (\sK)$
acts freely on both $G_1/K_1$ and  $\sG_1 / \sK_1$
so that $\widetilde{\sT_{i}}$
is $\Gamma_i$-equivariant.
It follows that the quotient map
\[
\sT_{i} = \widetilde{\sT_{i}} / \Gamma_i \, : \, G_1/K = (G_1/K_1)/\Gamma_i \longrightarrow \sG_1 / \sK = (\sG_1/\sK_1)/\Gamma_i 
\]
is a homotopy equivalence too.

\subsection{Surjectivity on Objects}
Let $\sG$ be a reductive group. It has a maximal compact subgroup $G$ by the Malcev-Iwasawa theorem \cite[Thm 32.5]{Str}.
The pull-back $\sR_{\sG}\rightarrow \sC_G$ 
gives a homomorphism of algebraic groups $f: \sT(G)\rightarrow \sG$. 

As we have already proved, $G$ is also a maximal compact subgroup of $\sT(G)$.
By the definition of $f$, the restriction $f|_G$ is the identity map on $G$. 
Look at the differential $df$ through the prism of the Malcev-Iwasawa decomposition:
\[ df : \fg \oplus i \fg = \mbox{Lie} ( \sT (G)) \longrightarrow \fg \oplus i \fg = \mbox{Lie} (G) \, .\]
Since $f|_G=\text{Id}_G$, $df|_{\fg}=\text{Id}_{\fg}$.
It follows that $df$ is the identity and $f$ is an isomorphism.

\section{Complexification and Subgroups}
Let $H\in\bS(G)$. We define $\tau$ by $\tau (H) \coloneqq \im  (\sT(f))$, where $f:H \rightarrow G$ is the inclusion map.  

\subsection{Subgroups $H$ and $\sH=\tau (H)$}
We  prove that $\sT(f)$ is injective and the first part follows directly from this fact.  Suppose $\sT(f)$ is not injective. Then it has a non-trivial kernel $\sK$, which cannot be unipotent.
So $\sK$ contains a non-trivial maximal compact subgroup $K$. Some conjugate of $hKh^{-1}$ for $h\in \sH$ must be a subgroup of $H$.
Since $\sT(f)|_{H}=f$, we find a non-trivial element in the kernel of $f$, a contradiction.

\subsection{Bijection} Since $\tau (xHx^{-1}) = x\tau (H)x^{-1}$, the map  $\overline{\tau}$ is injective.

Now pick $\sK\in\bS(\sG)$ and its maximal compact subgroup $K$. It is contained in a maximal compact subgroup $G^\prime$ of $\sG$.
Since maximal compact subgroups are conjugate, 
$G=gG^{\prime} g^{-1}$ for some $g\in\sG$. Hence, $gKg^{-1} \in \bS(g)$ and
$\overline{\tau} ([gKg^{-1}]) = [\sK]$.

\subsection{Homotopy Equivalence}
We need to show that the map
$\eta: G/H \rightarrow \sG / \sH$ given by $\eta (gH) \mapsto g\sH$
is a homotopy equivalence.
If $G$ is connected, this is the argument in Section~\ref{homotopy_equiv}. 
In general, any connected component of $\sG/\sH$ has a form $\sG_1 g \sH/\sH$ for some $g\in\sG$.
Since $\pi_0(G)=\pi_0 (\sG)$, there exists  $x\in H \cap \sG g$ so that
$\sG_1 g \sH/\sH = \sG_1 x \sH/\sH= \sG_1 / (x \sH x^{-1})$ and
the same argument settles the general case.

\subsection{Centralisers of Involutions}
If $x\in G$ is an involution and $H=C_G(x)$, then clearly 
$\tau (H) = C_{\sG}(x)$. 

  Let $\tau (H) = C_{\sG}(g)$ for some involution $g\in G$.
The finite group $\langle g \rangle \in \bS (\sG)$ is conjugate to $\tau (\langle x\rangle )$ for some involution $x\in G$.
By Lemma~\ref{centraliser}, the maximal compact subgroup of $C_{\sG}(x)$ is $C_{G}(x)$.
Moreover, the reductive subgroups $\tau (H)$ and $C_{\sG}(x)$ are conjugate.
Hence, their maximal compact subgroups $H$ and $C_{G}(x)$ are conjugate, so $H$ is the centraliser of an involution.

\section{Further Comments}
\subsection{Tannakian Formalism and Compact Groups}
Reconstruction of a compact group is the topic of the Tannaka-Krein Reconstruction \cite{DoRo}. 
Starting with $\sG$, this requires a $\star$-tensor category structure on $\sR_{\sG}$.
Given a representation $(V,\rho)$, the data defines a representation $(\overline{V},\overline{\rho})$ on the conjugate vector space
$\overline{V}=V$ via the action $\alpha \cdot v = \overline{\alpha} v$.
Given an automorphism $(x_V)\in\aut (F_{\sG})$, we define its $\star$-conjugate by
$$(x_V)^\star \coloneqq (y_V) \ \mbox{ where } \ y_{\overline{V}} = x_V \, .$$
This reconstructs a compact group $\sT_c (\sG) = \{ x\in \aut (F_{\sG}) \,\mid\, x^\star = x\}$.
However, there is no canonical $\star$-structure on $\sR_{\sG}$, which can be constructed by choosing a maximal compact subgroup $G$.
For example, choose two hermitian forms on $\bC^2$ that give different unitary subgroups of $\sG=\GL_2(\bC)$.
The corresponding $\star$-structures on $\sR_{\sG}$ will be different. 

\subsection{Intersections}
Consider two subgroups $\sH,\sK\in\bS(\sG)$ and their maximal compact subgroups $H$ and $K$.
Wouldn't it be useful to have some relation between the intersections $\sH\cap\sK$ and $H\cap K$?
Notice that  $\sH\cap\sK$ does not determine $H\cap K$, or vice versa.
Yet some limited information, for instance, about the generic behaviour of the other intersection is possible.

Categorically, the intersection  $H\cap K$ is a pull-back diagram:
\[
\begin{CD}
H\cap K  @>>> K\\
@VVV @V{\subseteq}VV\\
H @>{\subseteq}>> G
\end{CD}
\]
Note that $\Lie$ and $\Red$ are not equivalent and $\sT$ does not preserve pull-back diagrams. Hence 
$\sH\cap\sK$ and $H\cap K$ are not tightly related.
On the other hand, 
the homotopy equivalence of topological categories 
tells us that $\sT$ preserves the homotopy pull-back diagrams.
In other words, the corresponding pull-backs, 
the homotopy intersection groups
$H\cap^h K$ and $\sH \cap^h \sK$ are homotopy equivalent.
Recall that the standard model of the homotopy intersection is the following subgroup of the group of continuous paths
with pointwise multiplication: 
$$
H\cap^h K \coloneqq \{ \gamma: [0,1]\rightarrow G \,\mid\, \gamma(0)\in H, \; \gamma (1) \in K \} \, .
$$

Our motivation for studying intersections comes from our interest in the relation between
the locally symmetric spaces $K\backslash G/H$ and $\sK\backslash \sG/\sH$. This relation is beyond the scope of the present paper, but we make the following useful observation.

\begin{prop} \label{intersect}
If   $H,K \in \bS(G)$, then $\tau (H\cap K) = \tau (H) \cap \tau (K)$.
\end{prop}
\begin{proof}
  Let $\wfg$ be the Lie algebra of $\sG$, and $\fg$ be its Lie subalgebra corresponding to the Lie algebra of $G$.
  For $H,K\in\bS(G)$, let us denote their corresponding Lie subalgebras by $\fh$ and $\fk$. The  Malcev-Iwasawa  map $\mu$, defined in \eqref{MI_decomp},
  gives an explicit
  formula  for the map $\tau$
  \[ \tau (H) = \mu (H\times i \fh) = \{ \exp(\beta) h, \mid\, h\in H, \beta \in i \fh\} .\]
  The proposition follows from this formula: both  $\tau (H\cap K)$ and $\tau (H) \cap \tau (K)$
  are equal to
  \[ \{ \exp(\beta)h\, \mid\, h\in H\cap K, \ \beta \in i (\fh\cap\fk)\} \]
  since the presentation of an element of $G$ as $\mu(g,\alpha) = \exp(\alpha)g$ is unique.
\end{proof}
We say that a reductive subgroup $\sH\in\bS(\sG)$ is {\em $G$-controlled}, if it is of the form $\tau (H)$ for some $H\in\bS(G)$. 
The next two corollaries are useful properties of 
the intersections of $G$-controlled subgroups. 
\begin{cor}
  $H\cap K$ is a maximal compact subgroup of $\tau (H) \cap \tau (K)$.
\end{cor}
\begin{cor}
The intersection of two $G$-controlled reductive subgroups is reductive.
\end{cor}

\subsection{All Compact Groups}
Every compact group is a projective limit of compact Lie groups.
This suggests an extension of the functor
$\sT : \Lie 
\rightarrow
\Red$
to a functor from compact groups to pro-reductive groups.

Tannaka reconstruction is an alternative way to think of such a functor.
If $G$ is a compact group, we get an affine group scheme $\sT (G)$.
Is this a weak equivalence of topological categories?



\begin{thebibliography}{9999}
\bibitem{BaRa} A. Barut, R. Raczka, {\em Theory of group representations and applications}, Polish Scientific Publishers, 1980.
\bibitem{Bri} M. Brion,
  {\em Homomorphisms of algebraic groups: representability and rigidity},
	Michigan Math. J. {\bfseries  72} (2022), 51--76. 
      \bibitem{BLS}  T. Burness, M. Liebeck, A.  Shalev, 
        {\em The length and depth of compact Lie groups},
        Mathematische Zeitschrift, {\bfseries 294} 
        (2019), 1457--1476.
\bibitem{Che}C. Chevalley, {\em Theory of Lie Groups I}, Princeton Math. Ser., vol. 8. Princeton University Press, 1946.
\bibitem{DM} P. Deligne, J. S. Milne, {\em Tannakian categories}, Hodge Cycles, Motives, and Shimura Varieties, LNM 900, 1982, pp. 101--228.
\bibitem{DoRo} S. Doplicher, J. Roberts, 
{\em A new duality theory for compact groups},
{Invent. Math.} {\bfseries 98} (1989), 
157--218.
\bibitem{HH} W. Hazod, K. Hofmann, H.-P. Scheffler, M. W\"{u}stner, H.  Zeuner, 
{\em Normalizers of compact subgroups, the existence of commuting automorphisms, and applications to operator semistable measures}, 
J. Lie Theory {\bfseries 8(2)} (1998), 
189--209.
\bibitem{Ho0} G. P. Hochschild, {\em The structure of Lie groups}, Holden-Day Inc., 1965.
\bibitem{Ho}  G. P. Hochschild, {\em Basic Theory of Algebraic Groups and Lie Algebras},  
  Springer, 1981. 
\bibitem{Lur} J. Lurie, {\em Higher Topos Theory}, Annals of Math. Studies 170, Princeton University Press, 2009. 
\bibitem{RuTa} D. Rumynin, J. Taylor,
  {\em Real representations of $C_2$-graded groups: the linear and hermitian theories},
  {Higher Structures}, {\bfseries 6(1)} (2022), 
  359--374.
\bibitem{Str} M. Stroppel, 
{\em Locally Compact Groups}, 
EMS Textbooks in Math., EMS,  2006.
\end{thebibliography}
\end{document}